\documentclass[final]{amsart}
\usepackage{amsmath}
\usepackage{amssymb}
\usepackage{amsmath, amssymb, amsfonts, amsxtra}
\theoremstyle{proclaim}
\newtheorem{theorem}{Theorem}[section]
\newtheorem{lemma}[theorem]{Lemma}
\newtheorem{corollary}[theorem]{Corollary}
\newtheorem{proposition}[theorem]{Proposition}

\theoremstyle{fancyproclaim}

\theoremstyle{statement}
\newtheorem{remark}[theorem]{Remark}
\newtheorem{definition}[theorem]{Definition}

\theoremstyle{fancystatement}

\numberwithin{equation}{section}

\providecommand{\AMS}{$\mathcal{A}$\kern-.1667em%
	\lower.25em\hbox{$\mathcal{M}$}\kern-.125em$\mathcal{S}$}
\usepackage{xcolor}
\begin{document}% Do not forget this command!
	\title[Characterization  of extreme contractions  ]{Characterization  of extreme contractions through $k-$smoothness of operators   }
	\author[Arpita Mal, Kallol Paul and Subhrajit Dey]{Arpita Mal, Kallol Paul and Subhrajit Dey}

	\newcommand{\acr}{\newline\indent}

	\address[Mal]{Department of Mathematics\\ Jadavpur University\\ Kolkata 700032\\ West Bengal\\ INDIA}
	\email{arpitamalju@gmail.com}
	
	\address[Paul]{Department of Mathematics\\ Jadavpur University\\ Kolkata 700032\\ West Bengal\\ INDIA}
	\email{kalloldada@gmail.com}
	
	\address[Dey]{Department of Mathematics\\ Muralidhar Girls' College\\ Kolkata 700029\\ West Bengal\\ INDIA}
	\email{subhrajitdeyjumath@gmail.com}

	\thanks{The research of  Arpita Mal is supported by UGC, Govt. of India.  The research of Prof. Paul  is supported by project MATRICS(MTR/2017/000059)  of DST, Govt. of India. } 
	

\maketitle
\begin{abstract}
  We characrterize extreme contractions defined between  \ finite-dimensional polyhedral Banach spaces using  $k$- smoothness of  operators. We also explore weak L-P property, a recently introduced concept in the study of extreme contractions. We obtain a sufficient condition for a pair of finite-dimensional polyhedral Banach spaces to satisfy weak L-P property. As an application of these results, we explicitly compute the number of extreme contractions in some special Banach spaces. Our approach in this paper in studying extreme contractions lead to the improvement and generalization of  previously known results. 
\end{abstract}

\section{Introduction}
The study of extreme contractions and smoothness of operators between Banach spaces are two classical and fertile  areas of research in Banach space theory. While the characterization of extreme contractions defined between Hilbert spaces is well known \cite{G1,K,N,S},  the characterization of the same is still elusive, in the general setting of Banach spaces.  There are several papers including \cite{BR,CM,CM2,G,I,Ki,L1,L,LP,RRBS,SRP,Sh,S1},  that deal with the study of extreme contractions of operators defined between some special Banach spaces. In these papers, the authors have studied extreme contractions defined between particular Banach spaces. Recently in \cite{SPM}, the authors have characterized extreme contractions defined between two-dimensional strictly convex, smooth Banach spaces. From all these research works, it is clear that the characterization of extreme contractions depends heavily on the geometry of domain and range spaces. The purpose of this paper is to study extreme contractions between polyhedral Banach spaces and explore some interesting connections between the order of smoothness of an operator and extreme contraction.  Before proceeding further, we first establish the notations and terminologies.

We denote the Banach spaces by the letters $\mathbb{X}$ and $\mathbb{Y}.$ Throughout the paper, we assume that the Banach spaces are real.  $|A|$ denotes the cardinality of a set $A.$  An element $x$ of a convex set $A$ is said to be an extreme point of $A,$ if $x=ty+(1-t)z$ for some $t\in (0,1)$ and $y,z\in A$ implies that $x=y=z.$ The set of all extreme points of a convex set $A$ is denoted by $Ext(A).$ The unit ball and the unit sphere of $\mathbb{X}$ are denoted by $B_\mathbb{X}$ and $S_\mathbb{X}$ respectively, that is, $B_\mathbb{X}=\{x\in \mathbb{X}:\|x\|\leq 1\}$ and $S_\mathbb{X}=\{x\in \mathbb{X}:\|x\|=1\}.$ $\mathbb{L}(\mathbb{X},\mathbb{Y})$ denotes the space of all bounded linear operators defined from $\mathbb{X}$ to $\mathbb{Y}.$ $M_T$ denotes the set of all unit vectors at which $T$ attain its norm, that is, $M_T=\{x\in S_\mathbb{X}:\|Tx\|=\|T\|\}.$ For $x_1,x_2\in \mathbb{X},$ $L[x_1,x_2],L(x_1,x_2)$ and $L[x_1,x_2[$ represents the following sets:
$$L[x_1,x_2]=\{tx_1+(1-t)x_2:0\leq t\leq 1\},$$
$$L(x_1,x_2)=\{tx_1+(1-t)x_2:0< t< 1\}~\text{and}$$
$$L[x_1,x_2[=L[x_1,x_2]\cup \{tx_2-(t-1)x_1: t> 1\}.$$
$\mathbb{X}^*$ denotes the dual space of $\mathbb{X}.$ A bounded linear functional $x^*\in S_{\mathbb{X}^*}$ is said to be a supporting linear functional of a non-zero vector $x\in \mathbb{X},$ if $x^*(x)=\|x\|.$ For $x\in S_\mathbb{X},$ the set of all supporting linear functionals  of $x$ is denoted by $J(x),$ that is, $J(x)=\{x^*\in S_{\mathbb{X}^*}:x^*(x)=1\}.$ Note that, $J(x)$ is a non-empty, weak*-compact, convex subset of $S_{\mathbb{X}^*}.$ $x\in S_\mathbb{X}$ is said to be a smooth point if $J(x)$ is singleton. $x\in S_\mathbb{X}$ is said to be a $k-$smooth point \cite{KS} or the order of smoothness of $x$ is said to be $k$, if $J(x)$ contains exactly $k$ linearly independent functionals, that is, $k=\dim~span~J(x).$ From \cite[Prop. 2.1]{LR}, we get, if $x$ is $k-$smooth, then $k=\dim~span~Ext~J(x).$ Likewise an operator $T\in S_{\mathbb{L}(\mathbb{X},\mathbb{Y})}$ is said to be $k-$smooth operator, if $k=\dim~span~J(T)=\dim~span~Ext~ J(T).$  For more information on $k-$smoothness in Banach space, the readers may go through \cite{H,Ha,KS,LR,MP,W}. An operator $T\in \mathbb{L}(\mathbb{X},\mathbb{Y})$ is said to be an extreme contraction, if $T$ is an extreme point of the unit ball of $\mathbb{L}(\mathbb{X},\mathbb{Y}).$ Observe that, if $T$ is an extreme contraction, then $\|T\|=1.$ 
 Recall that, a finite-dimensional Banach space is said to be a polyhedral Banach space, if the unit ball contains only finitely many extreme points. Equivalently, a finite-dimensional Banach space is said to be polyhedral if $B_\mathbb{X}$ is a polyhedron.  In particular, a two-dimensional polyhedral Banach space is said to be a polygonal Banach space. Let us recall the following two definitions related to polyhedral Banach spaces.
\begin{definition}
	A polyhedron $P$ is a non-empty compact subset of $\mathbb{X}$ which is the intersection of finitely many closed half-spaces of $\mathbb{X},$ that is, $P=\cap_{i=1}^rM_i,$ where $M_i$ are closed half-spaces in $\mathbb{X}$ and $r\in \mathbb{N}.$ The dimension $dim(P)$ of the polyhedron $P$ is defined as the dimension of the subspace generated by the differences $v-w$ of vectors $v,w\in P.$
\end{definition}

\begin{definition}
	A polyhedron $Q$ is said to be a face of the polyhedron $P$ if either $Q=P$ or if we can write $Q=P\cap \delta M,$ where $M$ is a closed half-space in $\mathbb{X}$ containing $P$ and $\delta M$ denotes the boundary of $M.$ If $dim(Q)=i,$ then $Q$ is called an $i-$face of $P$. $(n-1)-$faces of $P$ are called facets of $P$ and $1-$faces of $P$ are called edges of $P.$
\end{definition}       

Motivated by the work of Lindenstrauss and Perles in \cite{LP}, the following two definitions have been introduced recently in \cite{RRBS,SRP}, to study  extreme contractions.

\begin{definition}\cite{SRP}
	Let $\mathbb{X},\mathbb{Y}$ be Banach spaces. We say that the pair $(\mathbb{X},\mathbb{Y})$ has L-P (abbreviated from Lindenstrauss-Perles) property if $T\in S_{\mathbb{L}(\mathbb{X},\mathbb{Y})}$ is an extreme contraction if and only if $T(Ext(B_\mathbb{X}))\subseteq Ext(B_\mathbb{Y}).$
\end{definition}

\begin{definition}\cite{RRBS}
	Let $\mathbb{X},\mathbb{Y}$ be Banach spaces. We say that the pair $(\mathbb{X},\mathbb{Y})$ has weak L-P  property if for each extreme contraction $T\in S_{\mathbb{L}(\mathbb{X},\mathbb{Y})},$ $T(Ext(B_\mathbb{X}))\cap Ext(B_\mathbb{Y})\neq \emptyset.$
\end{definition}

In this paper, we first obtain a characterization of extreme contractions defined between finite-dimensional polyhedral Banach spaces in terms of $k-$smoothness of the operators. As an immediate application of this result, we characterize the extreme contractions defined between two-dimensional polygonal Banach spaces. Next, we obtain a sufficient condition for a pair of finite-dimensional polyhedral Banach spaces to satisfy weak L-P property. This result generalizes \cite[Th. 2.1]{RRBS} and also improves on \cite[Th. 2.2]{RRBS}. Then we show that the sufficient condition for a pair $(\mathbb{X},\mathbb{Y})$ to satisfy weak L-P property, given in  Theorem \ref{th-weaklp} of this paper, is also a necessary condition, if $\mathbb{X}$ is two-dimensional polygonal Banach space and $\mathbb{Y}=\ell_{\infty}^2.$ However, by exhibiting proper examples, we show that this is not true in general. As a final result of this paper, we explicitly compute the exact number of extreme contractions defined on $\mathbb{X},$ where $S_\mathbb{X}$ is a regular hexagon. All the results obtained here, highlight the pivotal role of the order of smoothness of an operator in the study of extreme contractions defined between finite-dimensional polyhedral Banach spaces.

We will use \cite[Lemma 3.1]{Wa} in describing the structure of $Ext~J(T)$. For simplicity we state the lemma for finite-dimensional Banach spaces.

\begin{lemma}\cite[Lemma 3.1]{Wa}\label{lemma-wojcik}
	Suppose that $\mathbb{X},\mathbb{Y}$ are finite-dimensional Banach spaces. Let $T\in S_{\mathbb{L}(\mathbb{X},\mathbb{Y})}.$  Then $M_T\cap Ext(B_\mathbb{X})\neq \emptyset$ and 
	$$Ext ~J(T)=\{y^*\otimes x\in \mathbb{L}(\mathbb{X},\mathbb{Y})^*:x\in M_T\cap Ext(B_{\mathbb{X}}), y^*\in Ext ~J(Tx)\},$$
	where  $y^*\otimes x: \mathbb{L}(\mathbb{X},\mathbb{Y})\to \mathbb{R}$ is defined by $y^*\otimes x(S)=y^*(Sx)$ for every $S\in \mathbb{L}(\mathbb{X},\mathbb{Y}).$
\end{lemma}

\section{main results}

We begin this section with the characterization of exposed points of the unit ball of a finite-dimensional polyhedral Banach space, which clearly follows from 
\cite[Th. 3.5]{MP} and \cite[Th. 4.2]{W}.  Recall that an element $x \in S_{\mathbb{X}} $ is said to be an exposed point of the unit ball $ B_{\mathbb{X}},$ if there exists a supporting linear functional $ x^*$ of $x$ such that $x^*$ attains norm only at $\pm x.$ We also observe that   in  a finite-dimensional polyhedral Banach space, a point is an extreme point of the unit ball if and only if it is an exposed point of the same. We write these observations in the form of the following theorem. 
\begin{theorem}\label{th-ep}
	Let $\mathbb{X}$ be a polyhedral Banach space of dimension $n.$ Let $x\in S_{\mathbb{X}}.$ Then the following are equivalent:\\
	(a) $x$ is an exposed point of $B_\mathbb{X}.$\\
	(b) $x$ is an extreme point of $B_\mathbb{X}.$\\
	(c) $x$ is $n-$smooth. 
\end{theorem}

In the next theorem, we prove a characterization of extreme contractions defined between finite-dimensional polyhedral Banach spaces in terms of $k-$smoothness.

\begin{theorem}\label{th-ec}
	Let $\mathbb{X},\mathbb{Y}$ be polyhedral Banach spaces such that $\dim(\mathbb{X})=n$ and $\dim(\mathbb{Y})=m.$ Then $T\in S_{\mathbb{L}(\mathbb{X},\mathbb{Y})}$ is an extreme contraction if and only if $T$ is $mn-$smooth.
\end{theorem}
\begin{proof}
	Since $\mathbb{X},\mathbb{Y}$ are finite-dimensional Banach spaces, from \cite[Th. 1]{LO}, we get 
	\[Ext(B_{\mathbb{L}(\mathbb{X},\mathbb{Y})^*})=Ext(B_{\mathbb{Y}^*})\otimes Ext(B_{\mathbb{X}}).\]
	Since $Ext(B_{\mathbb{Y}^*})$ and  $Ext(B_{\mathbb{X}})$ are finite sets, there are only finitely many extreme points in the unit ball of $\mathbb{L}(\mathbb{X},\mathbb{Y})^*.$ Therefore, $\mathbb{L}(\mathbb{X},\mathbb{Y})^*$ is a polyhedral Banach space. Moreover, $\mathbb{L}(\mathbb{X},\mathbb{Y})^*$ is a finite-dimensional Banach space. Therefore, $\mathbb{L}(\mathbb{X},\mathbb{Y})$ is also a finite-dimensional polyhedral Banach space. Now, $\dim(\mathbb{L}(\mathbb{X},\mathbb{Y}))=mn.$ Hence, from Theorem \ref{th-ep}, we can conclude that  $T\in S_{\mathbb{L}(\mathbb{X},\mathbb{Y})}$ is an extreme contraction if and only if $T$ is $mn-$smooth. 
\end{proof}

Using Theorem \ref{th-ec}, we can now characterize the extreme contractions between two-dimensional polygonal Banach spaces.

\begin{theorem}\label{th-2dimec}
	Let $\mathbb{X},\mathbb{Y}$ be polygonal Banach spaces such that $\dim(\mathbb{X})=\dim(\mathbb{Y})=2.$ Let $T\in S_{\mathbb{L}(\mathbb{X},\mathbb{Y})}.$ Then $T$ is an extreme contraction if and only if either of the following holds:\\
	(i) $M_T\cap Ext(B_\mathbb{X})=\{\pm x_1,\pm x_2\}$ and $Tx_1,Tx_2\in Ext(B_\mathbb{Y}).$\\
	(ii) $M_T\cap Ext(B_\mathbb{X})=\{\pm x_1,\pm x_2,\pm x_3\}$ and $|\{x_i:Tx_i\in Ext(B_\mathbb{Y}),1\leq i\leq 3\}|\geq 2.$\\
	(iii) $M_T\cap Ext(B_\mathbb{X})=\{\pm x_1,\pm x_2,\pm x_3\},$ $Tx_1\in Ext(B_\mathbb{Y}),$ $Tx_2,Tx_3\notin Ext(B_\mathbb{Y})$ and there exist edges $F,G$ of $B_\mathbb{Y}$ such that $Tx_2\in F,Tx_3\in G$ and $F\neq \pm G.$\\
	(iv)  $|M_T\cap Ext(B_\mathbb{X})|\geq 8$ and there exists $x\in M_T\cap Ext(B_\mathbb{X})$ such that $Tx\in Ext(B_\mathbb{Y}).$\\
	(v)  $|M_T\cap Ext(B_\mathbb{X})|\geq 8$ and for each $x\in M_T\cap Ext(B_\mathbb{X}),$ $Tx\notin Ext(B_\mathbb{Y}).$ Moreover, there exist $x_i\in M_T\cap Ext(B_\mathbb{X})$ and $y_i^*\in Ext~J(Tx_i)$ for $1\leq i\leq 4$ such that $x_2=ax_1+bx_3,x_4=cx_1+dx_3,y_2^*=\alpha_1y_1^*+\alpha_2y_3^*,y_4^*=\beta_1y_1^*+\beta_2y_3^*$ and $\beta_1\alpha_2 ad-\beta_2\alpha_1 bc\neq 0.$
\end{theorem}
\begin{proof}
	From Theorem \ref{th-ec}, we can say that $T$ is an extreme contraction if and only if $T$ is $4-$smooth. If $T$ is an extreme contraction, then from \cite[Th. 2.2]{SRP}, we get, $span(M_T\cap Ext(B_\mathbb{X}))=\mathbb{X},$ that is, $|M_T\cap Ext(B_\mathbb{X})|\geq 4.$  Observe that, if $|M_T\cap Ext(B_\mathbb{X})|< 4,$ then from \cite[Th. 2.2]{MP}, we can conclude that $T$ is not $4-$smooth. Therefore, if $T$ is $4-$smooth, then $|M_T\cap Ext(B_\mathbb{X})|\geq 4.$ Hence, we only assume that $|M_T\cap Ext(B_\mathbb{X})|\geq 4.$\\
	
	First let $|M_T\cap Ext(B_\mathbb{X})|=4.$ In this case, we show that $T$ is an extreme contraction if and only if $(i)$ holds. Let $M_T\cap Ext(B_\mathbb{X})=\{\pm x_1,\pm x_2\}$ for some $x_1,x_2\in S_\mathbb{X}.$ Clearly, $\{x_1,x_2\}$ is linearly independent. Therefore, from \cite[Th. 2.2]{MP}, we can conclude that $T$ is extreme contraction, that is, $T$ is $4-$smooth if and only if $Tx_1$ and $Tx_2$ are $2-$smooth. Thus, by Theorem \ref{th-ep}, $Tx_1,Tx_2\in Ext(B_\mathbb{Y}).$ Therefore, if $|M_T\cap Ext(B_\mathbb{X})|=4,$ then $T$ is an extreme contraction if and only if $(i)$ holds.  \\
	
    	Now, let  $|M_T\cap Ext(B_\mathbb{X})|=6.$ In this case, we show that $T$ is an extreme contraction if and only if either $(ii)$ or $(iii)$ holds. Let $M_T\cap Ext(B_\mathbb{X})=\{\pm x_1,\pm x_2,\pm x_3\}$ for some $x_1,x_2,x_3\in S_\mathbb{X}.$ In this case, from \cite[Th. 3.1]{MP}, we get, $T$ is an extreme contraction, that is, $T$ is $4-$smooth if and only if $Tx_1$ is non-smooth, $Tx_2,Tx_3$ are not interior point of same line segment of $S_\mathbb{Y}$ and  $Tx_2,-Tx_3$ are not interior point of same line segment of $S_\mathbb{Y}.$ Since $Tx_1$ is non-smooth, $Tx_1$ is $2-$smooth. Hence, $Tx_1\in Ext(B_\mathbb{Y}).$ Now, if at least one of $Tx_2,Tx_3\in Ext(B_\mathbb{Y}),$ then $(ii)$ holds. Suppose $Tx_2,Tx_3\notin Ext(B_\mathbb{Y}).$ Then $Tx_2,Tx_3$ are interior points of line segments of $S_\mathbb{Y}.$ So there exist edges $F,G$ of $B_\mathbb{Y}$ such that $Tx_2\in F,Tx_3\in G.$ Since $Tx_2,Tx_3$ are not interior point of same line segment of $S_\mathbb{Y}$ and  $Tx_2,-Tx_3$ are not interior point of same line segment of $S_\mathbb{Y},$ we get that $F\neq \pm G.$ Thus, $(iii)$ holds. Therefore, if $|M_T\cap Ext(B_\mathbb{X})|=6,$ then $T$ is an extreme contraction if and only if either $(ii)$ or $(iii)$ holds. \\
    
    Next, let  $|M_T\cap Ext(B_\mathbb{X})|\geq 8 .$ Then from \cite[Th. 3.3]{MP}, we can easily conclude that $T$ is an extreme contraction if and only if either $(iv)$ or $(v)$ holds.
\end{proof}

We now turn our attention to the study of weak L-P property of a pair of Banach spaces. In the following proposition, we show that if $\mathbb{X}$ is a reflexive Banach space, then the pair $(\mathbb{X},\mathbb{R})$ satisfies weak L-P property.
\begin{proposition}\label{prop-weaklp}
	Let $\mathbb{X}$ be a reflexive Banach space. Then the pair $(\mathbb{X},\mathbb{R})$ satisfies weak L-P property. 
\end{proposition}
\begin{proof}
	Suppose $f\in S_{\mathbb{L}(\mathbb{X},\mathbb{R})}$ is an extreme contraction, that is, $f$ is an extreme point of $B_{\mathbb{X}^*}.$ Since $\mathbb{X}$ is a reflexive Banach space, $M_f\neq\emptyset.$ Now, $f$ must attain its norm at some extreme  point of $B_\mathbb{X}.$ Let $x\in M_f\cap Ext(B_\mathbb{X}).$ Then $|f(x)|=1.$ Thus, $f(x)$ is an extreme point of $B_\mathbb{R}.$ Hence, $f(Ext(B_\mathbb{X}))\cap Ext(B_\mathbb{R})\neq\emptyset.$ Therefore, the pair $(\mathbb{X},\mathbb{R})$ satisfies weak L-P property.   	
\end{proof}

In the next theorem, we obtain a sufficient condition for a pair of finite-dimensional polyhedral Banach spaces to satisfy weak L-P property.

\begin{theorem}\label{th-weaklp}
	Let $\mathbb{X},\mathbb{Y}$ be polyhedral Banach spaces and $\dim(\mathbb{X})=n,\dim(\mathbb{Y})=m.$ Let $|Ext(B_\mathbb{X})|=2(n+p).$ If $mp< n+p,$ then the pair $(\mathbb{X},\mathbb{Y})$ satisfies weak L-P property.
\end{theorem}
\begin{proof}
	Let $T\in S_{\mathbb{L}(\mathbb{X},\mathbb{Y})}$ be an extreme contraction. We show that there exists $x\in M_T\cap Ext(B_\mathbb{X})$ such that $Tx\in Ext(B_\mathbb{Y}).$  From \cite[Th. 2.2]{SRP},  we get $span(M_T\cap Ext(B_\mathbb{X}))=\mathbb{X},$ that is, $M_T\cap Ext(B_\mathbb{X})$ contains at least $n$ linearly independent elements. Let $M_T\cap Ext(B_\mathbb{X})=\{\pm x_i:1\leq i\leq r\}$ such that $\{x_1,x_2,\ldots, x_n\}$ is linearly independent. Then $r\leq n+p.$ If possible, suppose that $Tx_i\notin Ext(B_\mathbb{Y}),$ for any $i$,  $1\leq i\leq r.$  Then by Theorem \ref{th-ep}, $Tx_i$ is not $m-$smooth for each $1\le i\leq r.$ Let $Tx_i$ be $k_i-$smooth for all $1\leq i\leq r.$ Then $k_i\leq (m-1)$ for all $1\leq i\leq r.$ Clearly, $k_i=\dim~span~Ext~J(Tx_i).$
	Let $\{ y_{ij}^*\in Ext ~J(Tx_i):1\leq j\leq k_i\}$ be a basis of $span~Ext~J(Tx_i)$ for each $1\leq i\leq r.$  Let 
	\[W_i=span ~\{y^*\otimes x_i:y^*\in Ext ~J(Tx_i)\} ~\text{for each } 1\leq i\leq r.\]
	We first show that $B_i=\{y_{ij}^*\otimes x_i:1\leq j\leq k_i\}$ is a basis of $W_i.$ 
	Let
	$$\sum\limits_{1\leq j\leq m_i}a_j(y_{ij}^*\otimes x_i)=0, ~\text{where}~ a_j\in \mathbb{R}~\text{ for~ all~} 1\leq j\leq m_i.$$
	Choose a non-zero vector $y\in\mathbb{Y}.$ Define $S\in \mathbb{L}(\mathbb{X},\mathbb{Y})$ by 
	\begin{equation} \label{eq1}
	\begin{split}
	S x_i&=y\\
	S x_l&= 0 \ \forall \ l,  1\leq l(\neq i)\leq n. 
	\end{split}
	\end{equation}
	Now,
	\begin{eqnarray*}
		&&\sum\limits_{1\leq j\leq k_i}a_j(y_{ij}^*\otimes x_i)(S)=0\\
		&\Rightarrow&\sum\limits_{1\leq j\leq k_i}a_jy_{ij}^*S(x_i)=0\\
		&\Rightarrow& \sum\limits_{1\leq j\leq k_i}a_jy_{ij}^*(y)=0\\
		&\Rightarrow& \sum\limits_{1\leq j\leq k_i}a_jy_{ij}^*=0,~(\text{since~}y\in \mathbb{Y}~\text{is ~arbitrary})\\
		&\Rightarrow& a_j=0 ~\text{ for~ all~} 1\leq j\leq k_i.
	\end{eqnarray*}
	Thus, $B_i$ is linearly independent. It can be easily verified that $B_i$ is a spanning set of $W_i.$ Hence, $B_i$ is a basis of $W_i$  and so $\dim(W_i)=k_i$ for each $1\leq i\leq r.$ Now, let $T$ be $k-$smooth. Then 
	\begin{eqnarray*}
		k&=& \dim~span~J(T)\\
		&=& \dim~span~Ext~J(T)\\
		&=& \dim~span\{y^*\otimes x_i:y^*\in Ext~J(Tx_i),1\leq i\leq r\}\\
		&=& dim~W,\text{~where,}\\
		W&=&span ~\{y^*\otimes x_i:y^*\in Ext~ J(Tx_i),1\leq i\leq r \}.
	\end{eqnarray*}
	Clearly, $W\subseteq W_1+W_2+\ldots+W_r.$ Therefore,
	\[k=\dim(W)\leq\dim(\sum_{i=1}^{r}W_i)\leq \sum_{i=1}^{r}\dim(W_i)=\sum_{i=1}^{r}k_i\leq (m-1)r\leq(m-1)(n+p).\]
	Now, $mp< n+p$ implies that $(m-1)(n+p)<mn,$ that is,  $k<mn.$ Therefore, from Theorem \ref{th-ec}, we conclude that $T$ is not an extreme contraction. This is a contradiction. Thus, there exists $1\leq i\leq r$ such that $Tx_i\in Ext(B_\mathbb{Y}).$ Hence, the pair $(\mathbb{X},\mathbb{Y})$ satisfies weak L-P property. This completes the proof of the theorem.
\end{proof}

The following corollary now follows easily from Theorem \ref{th-weaklp}. 

\begin{corollary}\label{cor-weaklp}
	Let $\mathbb{X},\mathbb{Y}$ be  polyhedral Banach spaces such that $\dim(\mathbb{X})=n,\dim(\mathbb{Y})=m.$ Let $|Ext(B_\mathbb{X})|=2n+2$ and $m\leq n.$ Then the pair $(\mathbb{X},\mathbb{Y})$ satisfies weak L-P property.  
\end{corollary}

\begin{remark}
	(i) In  \cite[Th. 2.1]{RRBS}, Ray et al. proved that if $ \mathbb{X}$ is an $n$-dimensional  polyhedral Banach space with exactly $(2n+2)$ extreme points and $ m \leq n ,$ then the pair $( \mathbb{X}, \ell_{\infty}^n) $ satisfies weak L-P property. Clearly, Corollary \ref{cor-weaklp} improves on  \cite[Th. 2.1]{RRBS}.\\
	(ii) In \cite[Th. 2.2]{RRBS}, Ray et al. proved that if $\mathbb{X}$ is an $n-$dimensional polyhedral Banach space with exactly $(2n+2)$ extreme points and $m(m-1)\leq n,$ then the pair $(\mathbb{X},\ell_1^m)$ satisfies weak L-P property.  Observe that if $m>1$ and $m(m-1)\leq n,$ then $m\leq m(m-1) \leq n.$ Therefore, for $m>1,$ Corollary \ref{cor-weaklp} improves on \cite[Th. 2.2]{RRBS}.\\
	(iii)  Our Theorem  \ref{th-weaklp} unifies Theorems \cite[Th. 2.1]{RRBS} and \cite[Th. 2.2]{RRBS} and holds for the pair $(\mathbb{X}, \mathbb{Y}) $ with $\mathbb{Y}$ as $m$-dimensional polyhedral Banach space instead of the special Banach spaces $ \ell_{\infty}^m$ ( \cite[Th. 2.1]{RRBS} ) and $ \ell_{1}^m$ ( \cite[Th. 2.2]{RRBS} ).
\end{remark}

The natural question that arises now is whether the sufficient condition given in Theorem \ref{th-weaklp} for a pair of Banach spaces to satisfy weak L-P property is also a necessary condition or not. In the next theorem, we show that the condition is both necessary and sufficient  for the pair $(\mathbb{X},\ell_{\infty}^2),$ where $\mathbb{X}$ is two-dimensional polygonal Banach space.

\begin{theorem}\label{th-weaklpch}
	Let $\mathbb{X}$ be a two-dimensional polygonal Banach space. Then the pair $(\mathbb{X},\ell_{\infty}^2)$ satisfies weak L-P property if and only if $|Ext(B_\mathbb{X})|\leq 6.$ 
\end{theorem}
\begin{proof}
	Let $|Ext(B_\mathbb{X})|\leq 6.$ Then from Theorem \ref{th-weaklp}, we conclude that the pair $(\mathbb{X},\ell_{\infty}^2)$ satisfies weak L-P property. Conversely, suppose that $|Ext(B_\mathbb{X})|\geq 8.$ We show that the pair $(\mathbb{X},\ell_{\infty}^2)$ does not satisfy weak L-P property. Clearly, $\mathbb{X}^*$ is a two-dimensional polygonal Banach space such that $|Ext(B_{\mathbb{X}^*})|\geq 8.$ So we can choose $\{x_1,x_2,x_3,x_4\}\subseteq Ext(B_{\mathbb{X}^*})$ such that $L[x_1,x_3]$ and $L[x_1,-x_3]$ are not edges of $B_{\mathbb{X}^*}.$ Now, define $T\in \mathbb{L}(\ell_1^2,\mathbb{X}^*)$ by $Te_1=x_1$ and $Te_2=x_3,$ where $e_1=(1,0)$ and $e_2=(0,1).$ Then $M_T=\{\pm e_1,\pm e_2\}.$ Since $x_1,x_3$ are extreme points of $B_{\mathbb{X}^*}$ and $\mathbb{X}^*$ is polygonal Banach space, by Theorem \ref{th-ep}, $x_1,x_3$ are $2-$smooth points. Thus, by \cite[Th. 2.2]{MP}, $T$ is $4-$smooth. Hence, by Theorem \ref{th-ec}, $T$ is an extreme contraction. It is easy to observe that $T^*:\mathbb{X}\to \ell_{\infty}^2$ is also an extreme contraction. 	 We claim that $T^*(Ext(B_\mathbb{X}))\cap Ext(B_{\ell_{\infty}^2})=\emptyset.$  If possible, suppose that there exists $u\in Ext(B_\mathbb{X})$ such that $T^*u\in Ext(B_{\ell_{\infty}^2}).$ Then $T^*u$ is $2-$smooth.  Clearly,  $Ext~J(T^*u) \subset Ext(B_{\ell_1^2}) = \{\pm e_1, \pm e_2\}$ and so without loss of generality, we may assume that $Ext~J(T^*u)=\{e_1,e_2\}.$ Thus,
	\begin{eqnarray*}
		&& e_1(T^*u)=e_2(T^*u)=1\\
		&\Rightarrow & u(Te_1)=u(Te_2)=1\\
		&\Rightarrow& u(x_1)=u(x_3)=1\\
		&\Rightarrow& u(tx_1+(1-t)x_3)=1~\text{for~all}~t\in [0,1]\\
		&\Rightarrow& \|tx_1+(1-t)x_3\|=1~\text{for~all}~t\in [0,1].
	\end{eqnarray*}
	Hence, $L[x_1,x_3]$ is an edge of $B_{\mathbb{X}^*},$ a contradiction. Therefore, $T^*(Ext(B_\mathbb{X}))\cap Ext(B_{\ell_{\infty}^2})=\emptyset.$ Thus, if $|Ext(B_\mathbb{X})|\geq 8,$ then the pair $(\mathbb{X},\ell_{\infty}^2)$ does not satisfy weak L-P property. This completes the proof of the theorem.
\end{proof}

\begin{remark}
	Observe that, if we assume that $\mathbb{Y}$ is a two-dimensional polygonal Banach space such that $|Ext(B_\mathbb{Y})|=4,$ then similarly as in Theorem \ref{th-weaklpch}, we can show that the pair $(\mathbb{X},\mathbb{Y})$ satisfies weak L-P property if and only if $|Ext(B_\mathbb{X})|\leq 6.$
\end{remark}

In the next theorem, we show that,   given a  two-dimensional polygonal Banach space $\mathbb{X}$ with  $|Ext(B_\mathbb{X})|\geq 8,$ there exists a two-dimensional polygonal Banach space $ \mathbb{Y}$  with  $|Ext(B_\mathbb{Y})|\geq 6,$ such that  the pair $(\mathbb{X},\mathbb{Y})$ does not satisfy weak L-P property.

\begin{theorem}\label{th-weak-lp}
	Let $\mathbb{X}$ be a two-dimensional polygonal Banach space such that $|Ext(B_\mathbb{X})|\geq 8.$ Then for each $n\in\mathbb{N}\setminus\{1,2\},$ there exists a two-dimensional polygonal Banach space $\mathbb{Y}$ such that $|Ext(B_\mathbb{Y})|=2n$ and the pair $(\mathbb{X},\mathbb{Y})$ does not satisfy weak L-P property.
\end{theorem}
\begin{proof}
	Suppose $\{x_1,x_2,x_3,x_4\}\subseteq Ext(B_\mathbb{X})$ such that $x_i\neq \pm x_j$ for all $1\leq i,j\leq 4$ and $L[x_i,x_{i+1}]$ is an edge of $B_\mathbb{X}$ for each $1\leq i\leq 3.$ Let $L[x_1,x_2[~\cap L[x_4,x_3[=\{y_1\}$ and $L[x_2,x_1[~\cap L[-x_3,-x_4[=\{y_2\}.$ Let $z_1=\frac{y_1+x_2}{2}$ and $z_{n-1}=\frac{y_1+x_3}{2}.$ Now, for each $2\leq i\leq n-2,$ we can easily choose vectors $z_i=a_iz_1+b_iz_{n-1},$ where $a_i,b_i>0$ such that the convex hull of $\{\pm y_2,\pm z_j:1\leq j\leq n-1\}$ is a symmetric convex set with $2n$ extreme points $\{\pm y_2,\pm z_j:1\leq j\leq n-1\}$. Let $\mathbb{Y}$ be a Banach space such that $B_\mathbb{Y}$ is the convex hull of $\{\pm y_2,\pm z_j:1\leq j\leq n-1\}.$ Then $|Ext(B_\mathbb{Y})|=2n.$ It is clear that $x_i~(1\leq i\leq 4)$ are smooth points of $\mathbb{Y}.$ Now, consider the operator $I:\mathbb{X}\to\mathbb{Y}$ defined by $Ix=x$ for all $x\in \mathbb{X}.$ Then $M_I\cap Ext(B_\mathbb{X})=\{\pm x_1,\pm x_2,\pm x_3,\pm x_4\}.$ Observe that, $x_1,x_2\in L[z_1,y_2]$ and $x_3,x_4\in L[-y_2,z_{n-1}].$ So there exist $f,g\in S_{\mathbb{Y}^*}$ such that $J(x_1)=J(x_2)=\{f\}$ and $J(x_3)=J(x_4)=\{g\}.$ Suppose $I$ is $k-$smooth, where $1\leq k\leq 4.$ Then 
	\begin{eqnarray*}
		k&=&\dim~span~J(I)\\
		&=& \dim~span~Ext~J(I)	\\
		&=& \dim~span~\{f\otimes x_1,f\otimes x_2,g\otimes x_3,g\otimes x_4\}\\
		&=&4.	
	\end{eqnarray*}
	Therefore, $I$ is $4-$smooth. Hence, by Theorem \ref{th-ec}, $I$ is an extreme contraction. Now, observe that if $x\in Ext(B_\mathbb{X})\setminus M_I,$ then $\|Ix\|<\|I\|=1,$ that is, $x\notin Ext(B_\mathbb{Y}).$ If $x\in M_I\cap Ext(B_\mathbb{X}),$ then $Ix$ is smooth point of $B_\mathbb{Y}.$ Therefore, $I(Ext(B_\mathbb{X}))\cap Ext(B_\mathbb{Y})=\emptyset.$ Thus, the pair $(\mathbb{X},\mathbb{Y})$ does not satisfy weak L-P property. This completes the proof of the theorem.
\end{proof}

Although the above theorem indicates that the condition stated in Theorem \ref{th-weaklp} may be  necessary for arbitrary two-dimensional polygonal Banach spaces $\mathbb{X}, \mathbb{Y}$  to satisfy weak L-P property, the answer is still not known in its full generality. 
However, if $\dim(\mathbb{X}) > 2$, then the condition is not necessary.  We  exhibit  polyhedral Banach spaces $\mathbb{X},\mathbb{Y}$  that  satisfy weak L-P property but  $\dim(\mathbb{X})=n,\dim(\mathbb{Y})=m,|Ext(B_\mathbb{X})|=2(n+p)$ and $mp\geq n+p$ hold. To do so we  need the following two lemmas.

\begin{lemma}\label{lemma-4gon}
	Let $\mathbb{X}=\ell_{\infty}^{3}$ and $\mathbb{Y}$ be a two-dimensional Banach space.  Let $T \in S_{\mathbb{L}(\mathbb{X}, \mathbb{Y})}$ be such that $Rank(T)=2$ and $Ext(B_{\mathbb{X}}) \subseteq M_T.$ Then $T(B_{\mathbb{X}})$ is a convex set with $4$ extreme points.
\end{lemma}
\begin{proof}
	Let us consider the facets $\pm G_1$ of $B_{\mathbb{X}},$ where $G_1= \{(1,y,z) : |y|,|z| \leq 1 \}.$ Now, $G_1$ can be expressed as
	$$G_1= x + F,$$ where $x=(1,0,0)$ and $F=\{(0,y,z) : |y|,|z| \leq 1 \}.$ Then $$-G_1= -x + F.$$ It is clear that $B_{\mathbb{X}}$ is the convex hull of the sets $G_1=x+F$ and $-G_1=-x+F.$ Therefore, $T(B_{\mathbb{X}})$ is the convex hull of $T(x+F)=Tx+T(F)$ and $T(-x+F)=-Tx+T(F).$ Here, $F$ is the convex hull of $\{\pm(0,1,1),\pm(0,1,-1)\},$ that is, $T(F)$  is the convex hull of $\{\pm T(0,1,1),\pm T(0,1,-1)\}.$ Hence, $T(F)$ must be a symmetric set having at most four extreme points. If $T(F)$ has two extreme points say $\pm z,$ then the extreme points of $T(B_{\mathbb{X}})$ are $\pm Tx \pm z$ and we are done.\\
	Now, let $T(F)$ be a symmetric set having exactly four distinct extreme points $\pm T(0,1,1)$ and $\pm T(0,1,-1).$ We denote $T(0,1,1)$ and $T(0,1,-1)$ by $y_1$ and $y_2$ respectively. Clearly, $T(B_\mathbb{X})$ is the convex hull of $\{\pm Tx\pm y_1,\pm Tx\pm y_2\}.$ Since $\mathbb{Y}$ is two-dimensional and $\{y_1+y_2,y_1-y_2\}$ is linearly independent, we have, $Tx=a(y_1+y_2)+b(y_1-y_2),$ where $a,b\in \mathbb{R}.$ Now, the following cases may hold:\\
	(i) $a=b=0.$ (ii) $a=0,b\neq0,$ (iii) $a\neq 0,b=0,$ (iv) $a\neq 0,b\neq 0.$\\
	We consider each case separately.\\
	(i) Let $a=b=0.$ Then $Tx=0.$ Thus, $T(B_\mathbb{X})$ is the convex hull of $\{\pm y_1,\pm y_2\}.$ So the extreme points of $T(B_\mathbb{X})$ are $\pm y_1,\pm y_2$ and we are done.\\
	(ii) Let $a=0,b\neq 0.$ Then 
	$$Tx-y_1=\frac{2b}{2b+1}(Tx-y_2)+\frac{1}{2b+1}(-Tx-y_1)$$
	and 
	$$-Tx-y_2=\frac{1}{2b+1}(Tx-y_2)+\frac{2b}{2b+1}(-Tx-y_1).$$ 
	Thus, the only extreme points of $T(B_{\mathbb{X}})$ are $\pm (Tx-y_2)$ and $\pm (Tx+y_1).$ Hence, we are done.\\ 
	(iii) Let $a\neq0,b= 0.$ Then 
	$$Tx-y_1= \frac{2a}{2a+1}(Tx+y_2)+\frac{1}{2a+1}(-Tx-y_1)$$ and 
	$$-Tx+y_2=\frac{1}{2a+1}(Tx+y_2)+\frac{2a}{2a+1}(-Tx-y_1).$$
	Thus, the only extreme points of $T(B_{\mathbb{X}})$ are $\pm (Tx+y_2)$ and $\pm (Tx+y_1).$ Hence, we are done.\\
	(iv) Let $a\neq 0,b\neq 0.$  First let $a>0,b>0.$ Then 
	$$Tx-y_1=\frac{2a}{2a+2b+1}(Tx+y_2)+\frac{2b}{2a+2b+1}(Tx-y_2)+\frac{1}{2a+2b+1}(-Tx-y_1).$$
	Since $a>0,b>0,$ we have, $\frac{2a}{2a+2b+1},\frac{2b}{2a+2b+1},\frac{1}{2a+2b+1}\in (0,1).$ Moreover, we have, $\|Tx+y_2\|=\|Tx-y_2\|=\|Tx+y_1\|=1.$ Using this, it can be easily observed that $\|Tx-y_1\|<1,$ that is, $\|T(1,-1,-1)\|<1,$ which contradicts that $Ext(B_{\mathbb{X}}) \subseteq M_T.$ Similarly, considering $a<0,b<0,$ or $a<0,b>0$ or $a>0,b<0,$ we get a contradiction.  \\ 
	Thus, considering all cases, we get that $T(B_\mathbb{X})$ is a convex set with $4$ extreme points. This completes the proof of the lemma.  
	
\end{proof}

\begin{lemma}\label{lemma-04gon}
	Let $\mathbb{X}=\ell_{\infty}^{4}$ and $\mathbb{Y}$ be a two-dimensional Banach space.  Let $T \in S_{\mathbb{L}(\mathbb{X}, \mathbb{Y})}$ be such that $Rank(T)=2$ and $Ext(B_{\mathbb{X}}) \subseteq M_T.$ Then $T(B_{\mathbb{X}})$ is a convex set with $4$ extreme points.
\end{lemma}
\begin{proof}
	Let us consider the facets $\pm G_1$ of $B_{\mathbb{X}},$ where $G_1= \{(1,y,z,w) : |y|,|z|,|w| \leq 1 \}.$ Now, $G_1$ can be expressed as
	$$G_1= x + F,$$ where $x=(1,0,0,0)$ and $F=\{(0,y,z,w) : |y|,|z|,|w| \leq 1 \}.$ Then $$-G_1= -x + F.$$ Observe that, there exist $x_i\in F(1\leq i\leq 4)$ such that $\pm x+x_i\in Ext(B_\mathbb{X})$ and $F$ is the convex hull of $\{\pm x_i:1\leq i\leq 4\}.$ It is clear that $B_{\mathbb{X}}$ is the convex hull of the sets $G_1=x+F$ and $-G_1=-x+F.$ Therefore, $T(B_{\mathbb{X}})$ is the convex hull of $T(x+F)=Tx+T(F)$ and $T(-x+F)=-Tx+T(F).$ Here, $F$ is a symmetric cube about the origin. Similarly as in Lemma \ref{lemma-4gon}, it can be shown that $T(F)$ must be a symmetric set having at most six extreme points. If $T(F)$ has two extreme points say $\pm z,$ then the extreme points of $T(B_{\mathbb{X}})$ are $\pm Tx \pm z$ and we are done.\\
	Let $T(F)$ be a symmetric set having exactly six extreme points. Suppose the extreme points of $T(F)$ are $\pm Tx_1=\pm y_1, \pm Tx_2=\pm y_2$ and $\pm Tx_3=\pm y_3.$ Now, proceeding similarly as in  Lemma \ref{lemma-4gon}, it can be shown that $\pm Tx_4=\pm y_4$ are interior points of $T(F).$ Therefore, $Tx + Tx_4=Tx+y_4$ is also an interior point of $Tx+T(F)$ and hence $||Tx+y_4|| < 1.$ But $Tx+y_4=T(x+x_4),$ where $x+x_4 \in Ext(B_{\mathbb{X}}),$ which contradicts that $Ext(B_{\mathbb{X}}) \subseteq M_T.$ Thus, $T(F)$ can not have six extreme points.\\
	Suppose $T(F)$ has exactly four extreme points. Now, proceeding similarly as Lemma \ref{lemma-4gon}, we can show that $T(B_{\mathbb{X}})$ has exactly four extreme points. This completes the proof of the lemma.
\end{proof}

\begin{remark}
	Following the same line of arguments we can show that  Lemma \ref{lemma-04gon} holds for $\mathbb{X} = \ell_{\infty}^n,$ i.e., if  $T \in S_{\mathbb{L}(\ell_{\infty}^n, \mathbb{Y})}$ with  $Rank(T)=2$ and $Ext(B_{\mathbb{X}}) \subseteq M_T,$ then $T(B_{\mathbb{X}})$ is a convex set with $4$ extreme points.
\end{remark}

Next, we obtain a bound of the order of smoothness of a class of bounded linear operators defined between $\ell_{\infty}^4$ and a two-dimensional Banach space. 
\begin{theorem}\label{th-ksmooth}
	Let $\mathbb{X}=\ell_{\infty}^{4}$ and $\mathbb{Y}$ be any two-dimensional Banach space. Suppose $T \in S_{{\mathbb{L}(\mathbb{X},\mathbb{Y})}}$ is such that $Ext(B_{\mathbb{X}}) \subseteq M_T$ and $Tx$ is smooth for all $x \in Ext(B_{\mathbb{X}}).$ Then $T$ is $k-$smooth where $k \leq 6$.
\end{theorem}
\begin{proof}
Let us write $Ext(B_{\mathbb{X}})=\{\pm x_1, \pm x_2, \dots, \pm x_8 \},$ where $\{x_1, x_2, x_3, x_4 \}$ is linearly independent. Let $S= \{x_1, x_2, \dots, x_8 \}.$\\
First suppose $Rank(T)=1.$ Then there exists $y^* \in S_{\mathbb{Y}^*}$ such that for any $i \in \{1,2,\dots, 8 \},$ $J(Tx_i)= \{y^* \}$ or $\{-y^* \}.$  Now, if $T$ is $k-$smooth, then 
\begin{eqnarray*}
	k&=& dim ~span ~J(T)\\
	&=& dim~ span~ Ext ~J(T)\\
	&=& dim~ span ~\{y^*\otimes x_i : 1 \leq i \leq 8 \}\\
	&=& dim~ span ~\{y^*\otimes x_i : 1 \leq i \leq 4 \}\\
	&=& 4,
\end{eqnarray*}
as $\{y^*\otimes x_i : 1 \leq i \leq 4 \}$ is linearly independent by \cite[Lemma 2.1]{MP}. Hence $T$ is $4-$smooth.\\
 Let $Rank(T)=2.$ Then by Lemma \ref{lemma-04gon}, $T(B_\mathbb{X})$ is a convex set with four extreme points. Without loss of generality, let $\pm Tx_1,\pm Tx_2$ be four distinct extreme points of $T(B_\mathbb{X}).$ Suppose $Tx_i\in L(Tx_1,Tx_2)$ for some $3\leq i\leq 8.$ We claim that for each $3\leq i\leq 8, Tx_i\in L[Tx_1,Tx_2]\cup L[-Tx_1,-Tx_2].$ Since $\|Tx_i\|=1,L[Tx_1,Tx_2]\subseteq S_\mathbb{Y}.$ Let $J(Tx_i)=\{y^*\}.$ Then for each $y\in L[Tx_1,Tx_2],$ $y^*(y)=1.$ If possible, let there exist $3\leq j(\neq i)\leq 8$ such that $Tx_j\in L(Tx_1,-Tx_2).$ Let $J(Tx_j)=\{z^*\}.$ Then $y^*\neq \pm z^*.$ Now, for all $y\in L[Tx_1,-Tx_2],z^*(y)=1.$ Thus, $y^*,z^*\in J(Tx_1),$ contradicts that $Tx_1$ is smooth. Therefore, $Tx_j\notin L(Tx_1,-Tx_2).$ Similarly, it can be shown that $Tx_j\notin L(-Tx_1,Tx_2).$  Therefore, for any $i \in \{1,2,\dots, 8 \},$ $J(Tx_i)= \{y^* \}$ or $\{-y^* \}.$ Now, it is easy to observe that $T$ is $4-$smooth.\\
Now, suppose that $Tx_i\notin L(\pm Tx_1,\pm Tx_2)$ for any $3\leq i\leq 8.$ Then $Tx_i\in \{\pm Tx_1,\pm Tx_2\}$ for all $3\leq i\leq 8.$ Let $J(Tx_1)=\{y_1^*\}$ and $J(Tx_2)=\{y_2^*\}.$ Then for any $i \in \{1,2,\dots, 8 \},$
\begin{eqnarray*}
	&J(Tx_i)=& \{y_1^* \} ~or~ \{-y_1^* \} ~or~ ~\{y_2^* \}~ or~ \{-y_2^* \}.
\end{eqnarray*}
 Thus, there exist two subsets $S_1$ and $S_2 (S_1 \cap S_2 \neq \phi, S_1 \cup S_2 = S)$ of $S$ such that $T(S_1)= \pm Tx_1$ and $T(S_2)= \pm Tx_2.$ Therefore, we have for any $i \in \{1,2,\dots, 8 \},$
\begin{eqnarray*}
	J(Tx_i)&=& \{y_1^* \} ~ or ~ \{-y_1^* \}, ~\text{if~}x_i \in S_1\\
	&=& \{y_2^* \} ~ or ~ \{-y_2^* \},~\text{if~} x_i \in S_2.
\end{eqnarray*}
Now, it is clear that any $4$ elements of $S_1$ as well as $S_2$ are linearly dependent. Otherwise, if $\{ x_{11}, x_{12}, x_{13}, x_{14} \}$ is a linearly independent subset of $S_1,$ then for any $x \in \mathbb{X},$
\begin{eqnarray*}
	&x& =       \sum_{i=1}^{4} \lambda_{i}x_{1i}, ~ \textit{where} ~ \lambda_{i} ~\textit{are scalars,} \\
	&\Rightarrow& Tx= \sum_{i=1}^{4} \lambda_{i}Tx_{1i}\in span\{Tx_1\}.
\end{eqnarray*}
 Hence, $Rank(T)=1,$ a contradiction. Thus, maximal linearly independent subsets of $S_1$ and $S_2$ contain at most $3$ elements. Let us write those linearly independent subsets of $S_1$ and $S_2$ respectively by $A_1=\{x_{1i} : 1 \leq i \leq n_{1} \}$ and $A_2=\{ x_{2i} : 1 \leq i \leq n_{2} \},$ where $n_1, n_2 \leq 3.$ Now, if $T$ is $k-$smooth, then 
\begin{eqnarray*}
	k&=& dim~span~J(T)\\
	&=& dim~span~Ext ~J(T)\\
	&=& dim~span~\{y_1^*\otimes x, y_2^*\otimes z : x \in S_1, z \in S_2 \}\\
	&=& dim~span~\{y_1^*\otimes x_{1i}, y_2^*\otimes x_{2i} : 1 \leq i \leq n_{1}, 1 \leq j \leq n_{2}  \}\\
	&\leq&n_1+n_2\leq 6.
\end{eqnarray*}
Therefore, $T$ is $k-$smooth, where $k\leq 6.$ This completes the proof of the theorem.
\end{proof}

Now, we are in a position to show that although the pair $(\ell_{\infty}^4,\mathbb{Y})$ does not satisfy the  condition given in Theorem \ref{th-weaklp}, the pair $(\ell_{\infty}^4,\mathbb{Y})$ satisfies weak L-P property, where $\mathbb{Y}$ is a two-dimensional polygonal Banach space.

\begin{theorem}
  Let $\mathbb{Y}$ be a two-dimensional polygonal Banach space. Then the pair $(\ell_{\infty}^4,\mathbb{Y})$ satisfies weak L-P property.
\end{theorem}
\begin{proof}
Let $\mathbb{X}=\ell_{\infty}^{4}.$ Then  $|Ext(B_{\mathbb{X}})|=16=2(4+4).$ Observe that, comparing with Theorem \ref{th-weaklp}, here we have $m=2,n=4$ and $p=4.$ Thus, $mp < n+p$ is not satisfied. We now show that the pair $(\mathbb{X}, \mathbb{Y})$ satisfies weak L-P property. Let $T \in S_{\mathbb{L}(\mathbb{X},\mathbb{Y})}$ be an extreme contraction. First let us assume $|M_T \cap Ext(B_{\mathbb{X}})|=16.$ If $Tx \in Ext(B_{\mathbb{Y}})$ for some $x \in M_T \cap Ext(B_{\mathbb{X}}),$ then we are done. If possible, let $Tx \notin Ext(B_{\mathbb{Y}})$ for any $x \in M_T \cap Ext(B_{\mathbb{X}}).$ Then from Theorem \ref{th-ep}, we get $Tx$ is smooth for all $x \in M_T \cap Ext(B_{\mathbb{X}}).$ Now, from Theorem \ref{th-ksmooth}, we see that $T$ is $k-$smooth, where $k \leq 6.$ Hence, by Theorem \ref{th-ec}, $T$ is not an extreme contraction, a contradiction. Thus, $Tx \in Ext(B_{\mathbb{Y}})$ for some $x \in M_T \cap Ext(B_{\mathbb{X}}).$ 
Now, let $|M_T \cap Ext(B_{\mathbb{X}})| = 14$ and $Tx \notin Ext(B_{\mathbb{Y}})$ for any $x \in M_T \cap Ext(B_{\mathbb{X}}).$ Thus, $Tx$ is smooth for all $x \in M_T \cap Ext(B_{\mathbb{X}}).$ Let $M_T \cap Ext(B_{\mathbb{X}})=\{\pm x_1, \pm x_2, \dots, \pm x_7 \}$ and $J(Tx_i)= \{y_i^* \}, 1 \leq i \leq 7.$ Now, if $T$ is $k-$smooth, then 
\begin{eqnarray*}
	k&=& dim~span~J(T)\\
	&=& dim~span~Ext ~J(T)\\
	&=& dim~span~\{y_i^*\otimes x_i : 1 \leq i \leq 7 \}\\
	&\leq& 7.
\end{eqnarray*}
So, $T$ can not be $8-$smooth and hence not an extreme contraction, a contradiction. Therefore, $Tx \in Ext(B_{\mathbb{Y}})$ for some $x \in M_T \cap Ext(B_{\mathbb{X}}).$ Similarly, if $|M_T \cap Ext(B_{\mathbb{X}})| < 14,$  then we can show that $Tx \in Ext(B_{\mathbb{Y}})$ for some $x \in M_T \cap Ext(B_{\mathbb{X}}).$ Thus, the pair $(\mathbb{X}, \mathbb{Y})$ satisfies weak L-P property. This completes the proof of the theorem.
\end{proof}

As an immediate application of Theorem \ref{th-2dimec} (or Theorem \ref{th-weaklp}), we can compute the number of extreme contractions defined on $\mathbb{X},$ where $S_\mathbb{X}$ is a regular hexagon.

\begin{theorem}
	Let $\mathbb{X}$ be a two-dimensional polygonal Banach space such that $S_\mathbb{X}$ is a regular hexagon. Then $|Ext(B_{\mathbb{L}(\mathbb{X},\mathbb{X})})|=30.$
\end{theorem}
\begin{proof}
  Without loss of generality, we may assume that the vertices of $S_\mathbb{X}$ are $\pm x_1=\pm (1,0),\pm x_2=\pm (\frac{1}{2},\frac{\sqrt{3}}{2}),\pm x_3=\pm (-\frac{1}{2},\frac{\sqrt{3}}{2}).$ Let $T\in Ext(B_{\mathbb{L}(\mathbb{X},\mathbb{X})}).$ Then from Theorem \ref{th-2dimec}, we can say that either $(i)~|M_T\cap Ext(B_\mathbb{X})|=4$ or  $(ii)~|M_T\cap Ext(B_\mathbb{X})|=6.$\\
  
 $(i)$ First consider the case $|M_T\cap Ext(B_\mathbb{X})|=4.$ Let $M_T\cap Ext(B_\mathbb{X})=\{\pm x_1,\pm x_2\}.$ Then by Theorem \ref{th-2dimec}, we get $Tx_1,Tx_2\in Ext(B_\mathbb{X}).$ Observe that, if $x,y$ are two distinct extreme points of $B_\mathbb{X},$ then $\|x-y\|\geq 1.$ Now, $x_3=x_2-x_1$ and $x_3\notin M_T$ ensures that $Tx_1$ and $Tx_2$ cannot be distinct extreme points of $B_\mathbb{X}.$ Therefore, $Tx_1=Tx_2.$ Now, there are $6$ possibilities for $Tx_1.$ Hence, there are $6$ extreme contractions $T\in \mathbb{L}(\mathbb{X},\mathbb{X})$ such that $M_T\cap Ext(B_\mathbb{X})=\{\pm x_1,\pm x_2\}.$ Similarly, it can shown that there are $6$ extreme contractions $T\in \mathbb{L}(\mathbb{X},\mathbb{X})$ such that $M_T\cap Ext(B_\mathbb{X})=\{\pm x_2,\pm x_3\}.$ Now, suppose that $M_T\cap Ext(B_\mathbb{X})=\{\pm x_1,\pm x_3\}.$ Since $x_2=x_1+x_3$ and $x_2\notin M_T,$ $Tx_1\neq Tx_3.$ Observe that,  if $x,y$ are two linearly independent extreme points of $B_\mathbb{X},$ then $\|x+y\|\geq 1.$ Therefore, $Tx_1,Tx_3$ are not linearly independent. Hence, $Tx_1=-Tx_3.$ Now, there are $6$ possibilities for $Tx_1.$ Thus, there are $6$ extreme contractions $T\in \mathbb{L}(\mathbb{X},\mathbb{X})$ such that $M_T\cap Ext(B_\mathbb{X})=\{\pm x_1,\pm x_3\}.$ So we get $18$ extreme contractions $T$ such that $|M_T\cap Ext(B_\mathbb{X})|=4.$\\
  
 $(ii)$ Now, consider the case $|M_T\cap Ext(B_\mathbb{X})|=6,$ that is, $M_T\cap Ext(B_\mathbb{X})= \{\pm x_1,\pm x_2,$ $\pm x_3\}.$ We show that in this case, $T$ is an isometry. By Theorem \ref{th-2dimec} (or Theorem \ref{th-weaklp}), $Tx_i\in Ext(B_\mathbb{X})$ for some $1\leq i\leq 3.$ Without loss of generality, let $Tx_1\in Ext(B_\mathbb{X})$ and $Tx_1=x_1.$ Now, the following cases may hold:
 \begin{align*}
 (1)~Tx_3\in & L[x_1,x_2],  &(2)~Tx_3\in & L[x_2,x_3], & (3)~Tx_3\in L[x_3,-x_1],\\ (4)~Tx_3\in & L[-x_1,-x_2],  & (5)~Tx_3\in & L[-x_2,-x_3], & (6)~Tx_3\in L[-x_3,x_1].
 \end{align*}
 We consider each case separately.\\
 $(1)$ Let $Tx_3\in  L[x_1,x_2].$ Then $Tx_3=tx_1+(1-t)x_2,$ for some $t\in [0,1].$ Then $Tx_2=Tx_1+Tx_3=x_1+tx_1+(1-t)x_2=(1+t)x_1+(1-t)x_2=\Big(\frac{3}{2}+\frac{t}{2},(1-t)\frac{\sqrt{3}}{2}\Big).$ Thus, $\|Tx_2\|>1,$ a contradiction.\\
  $(2)$ Let $Tx_3\in  L[x_2,x_3].$ Then $Tx_3=tx_2+(1-t)x_3,$ for some $t\in [0,1].$ Thus, $Tx_2=tx_1+x_2=(t+\frac{1}{2},\frac{\sqrt{3}}{2}).$ Since $\|Tx_2\|=1,$ we must have $t=0,$ that is, $Tx_3=x_3.$\\
  $(3)$ Let $Tx_3\in L[x_3,-x_1]$. Then  $Tx_3=tx_3-(1-t)x_1$ for some $t\in [0,1].$ Thus,  $Tx_2=tx_2.$ Since $\|Tx_2\|=1,$ we have $t=1.$ Thus, $Tx_3=x_3.$\\
  $(4)$ Let $Tx_3\in  L[-x_1,-x_2].$ Then $Tx_3=-tx_1-(1-t)x_2$ for some $t\in [0,1].$ Thus, $Tx_2=-(1-t)x_3.$ Since $\|Tx_2\|=1,$ we have $t=0.$ Thus, $Tx_3=-x_2.$\\
  $(5)$ Let $Tx_3\in  L[-x_2,-x_3].$ Then $Tx_3=-tx_2-(1-t)x_3$ for some $t\in [0,1].$ Thus, $Tx_2=(1-t)x_1-x_3=\Big(\frac{3}{2}-t,-\frac{\sqrt{3}}{2}\Big).$ Since $\|Tx_2\|=1,$ we have $t=1.$ Thus, $Tx_3=-x_2.$\\
  $(6)$ Let $Tx_3\in L[-x_3,x_1].$ Then $Tx_3=tx_1-(1-t)x_3$ for some $t\in [0,1].$ Similarly as case $(1)$, we can show that $\|Tx_2\|>1,$ a contradiction.\\
  
   Therefore, if $Tx_1=x_1,$ then considering all possibilities for $Tx_3,$ we get that either $Tx_3=x_3$ or $Tx_3=-x_2.$ In each case, $T$ is an isometry. Clearly, an isometry is an extreme contraction. Now, it is easy to observe that there are $12$ isometries on $\mathbb{X}.$ Therefore, there are $12$ extreme contractions $T$ such that $|M_T\cap Ext(B_\mathbb{X})|=6.$  
  
 Combining $(i)$ and $(ii)$, we get total $18+12=30$ extreme contractions on $\mathbb{X}.$ 
\end{proof}

\bibliographystyle{amsplain}

\end{document}